\documentclass[11pt]{amsart}
\usepackage{amssymb, amsmath, amsthm}
\usepackage[latin1]{inputenc}
\usepackage{graphicx}
\usepackage[final]{hyperref}
\usepackage[a4paper, centering]{geometry}
\usepackage{color}
\usepackage{graphicx} 

\newtheorem{theorem}{Theorem}
\newtheorem{proposition}[theorem]{Proposition}
\newtheorem{lemma}[theorem]{Lemma}
\newtheorem{corollary}[theorem]{Corollary}
\theoremstyle{definition}
\newtheorem{definition}[theorem]{Definition}
\theoremstyle{remark}
\newtheorem{remark}[theorem]{Remark}

\def\R{\mathbb{R}}

\def\N{\mathbb{N}}

\def\pscal#1#2{\left\langle#1,\,#2\right\rangle}

\long\def\delete#1{}

\DeclareMathOperator{\Cut}{\overline \Sigma}  
\DeclareMathOperator{\cent}{C}
\DeclareMathOperator{\high}{M}

\font\maius=cmcsc10 scaled1200

\begin{document}

\title[Proximally smooth sets]%
{On the characterization of some classes \\ of proximally smooth sets}%
\author[G.~Crasta, I.~Fragal\`a]{Graziano Crasta,  Ilaria Fragal\`a}
\address[Graziano Crasta]{Dipartimento di Matematica ``G.\ Castelnuovo'', Univ.\ di Roma I\\
P.le A.\ Moro 2 -- 00185 Roma (Italy)}
\email{crasta@mat.uniroma1.it}

\address[Ilaria Fragal\`a]{
Dipartimento di Matematica, Politecnico\\
Piazza Leonardo da Vinci, 32 --20133 Milano (Italy)
}
\email{ilaria.fragala@polimi.it}



\date{April 4, 2013}

\begin{abstract} We provide a complete characterization of closed sets with empty interior and positive reach in $\R^2$. 
As a consequence, we characterize open bounded domains in $\R^2$ whose high ridge and cut locus agree, and hence  $C^1$ planar domains whose normal distance to the cut locus is constant along the boundary.  The latter results extends to convex domains in $\R^n$. \end{abstract}

\maketitle

\section{Introduction}\label{intro}

A nonempty closed subset $S$ of $\R^n$ is called {\it proximally smooth}, or with {\it positive reach}, if for every point $x$ belonging to an open tubular neighborhood outside $S$ there is a unique minimizer of the distance function from $x$ to $S$. 

These sets were introduced in 1959 in the seminal paper \cite{Fed1} by Federer, 
who also proved many of their most relevant properties, in particular the validity of a tube formula, which expresses the Lebesgue measure of a sufficiently small $r$-parallel neighborhood of a set with positive reach in $\R ^n$ as a polynomial in $r$ of degree $n$. 

The concept of proximal smoothness can in fact be located at the crossroad of different areas, such as Geometric Measure Theory, Convex Geometry, Nonsmooth Analysis, Differential Geometry.  Since Federer, it has been investigated and developed in various ways. Related research directions include 
generalized Steiner-type formulae, tubular neighborhoods, and curvature measures
\cite{CoHu,Hu,HuLaWe,RaZa,Za};
connections with Lipschitz functions, semi-concave functions, and lower-$C^2$ functions 
\cite{CSW,Fu};
proximal smoothness in abstract frameworks, such as Banach spaces or 
Riemannian manifolds \cite{Ba,BeThZl}; 
applications to nonlinear control systems and differential inclusions
\cite{CaSib,CoMaWo,CoNg}. 

More comprehensive accounts of results in this area and related bibliography can be found in the surveys papers \cite{CoTh,Tha}.

\medskip
In this paper we are concerned with the following question: 

\bigskip\centerline{\it (*)  Which is the geometry of a closed set $S \subset \R^2$ with positive reach and empty interior? }

\medskip
As far as we are aware, no previous contributions are available in this respect in the literature. In particular, it is worth advertising that one cannot apply the several existing results which allow to retrieve regularity information on a set starting from the regularity of its distance function (possibly squared or signed). Indeed, some of these results are classical and some others are more recent (see e.g. \cite{Amb2,BeMaNo,DZa,Gi}), but in any case they rely on some regularity assumption on the distance  {\it up to} the involved set. In spite, by definition, the distance function from a set $S$ of positive reach is required to be differentiable just on the set of points where it is sufficiently small and strictly positive, thus not necessarily on $S$ itself  (see Definition \ref{d:psm}). 

\smallskip
Our main results provide a complete answer to question (*): 
each connected component of $S$ is either a singleton or  a manifold of class $C ^ {1,1}$ (see 
Theorem \ref{t:c11}); in case the distance from $S$  is at least  $C^2$ in an open neighborhood of $S$,  then such manifolds have no boundary and are of class $C ^2$ (see 
Theorem \ref{t:c2}); moreover,  in case the distance from $S$ goes  beyond the $C ^2$ threshold, $S$ gains the same regularity (see Remark \ref{r:comments}). 

\smallskip
As a by-product, we are able to answer the following related question:

\medskip
\centerline{\it (**) Which is the geometry of a set $\Omega \subset \R^2$ whose high ridge and cut locus agree?}

\medskip

Recall that, given an open bounded domain $\Omega\subset \R ^2$, the {\it high ridge} is the set of points where the distance function from $\partial \Omega$ attains its maximum over $\overline \Omega$, while the {\it cut locus} is the closure in $\overline \Omega$ of the so-called {\it skeleton}, namely of the sets of points in $\Omega$ which admit multiple closest points on $\partial \Omega$;  recall also that the {\it central set} is formed by the centers of the maximal disks contained into $\Omega$.  
We refer to Section \ref{main} for the precise formulation of these definitions.  All these sets, which have each one its own role in the geometry of the distance function from the boundary, have been widely investigated in the literature, often with a non-uniform  terminology.  A miscellaneous collection of related references, without any attempt of completeness, is \cite{ACKS,Al,BiHa,CMf,Frem,FW,Lie,MM}. 
It must be added that recently the singular set of the distance function has raised an increasing interest also in applied domains, such as computer science and visual reconstruction, and this is especially true for the central set (often named medial axis in this context), see e.g.\ 
\cite{Bishop2012,CCM,Degen2004,Wol}
and Remark~\ref{r:MAT} below.


\smallskip
To pinpoint the link between questions (*) and (**), one has to observe that, if the cut locus and high ridge of a domain $\Omega$ coincide, they can be identified with a proximally smooth set with empty interior. As a consequence, the answer to question (**) is: $\Omega$ is the outer parallel neighborhood of a $C ^ {1,1}$ manifold; in particular, if $\Omega$ is assumed to be of class $C ^2$ and simply connected, it must be necessarily a disk (see Theorem \ref{forBK}). 

\smallskip
We remark that our answer to question (**) solves also the problem of characterizing domains of class $C ^1$ whose {\it normal distance} to the cut locus is constant along the boundary (see Corollary \ref{cor:lambda}).  Intuitively, the normal distance of a point $y \in \partial \Omega$ measures how far one can enter into $\Omega$ starting at $y$ and moving along the direction of the inner normal before hitting the cut locus; the precise definition is recalled in Section \ref{main}. This notion has been considered from different points of views: in \cite{CCG,IT,LN}
the regularity of the normal distance under different requirements on the boundary has been investigated, along with some applications to Hamilton-Jacobi equations and to PDEs related with granular matter theory;  in \cite{Ce1,CFGa,CFGb,CFGd} the normal distance has been exploited in order to study the minimizing properties of the so-called web functions.  
Let us also mention that, in a previous paper, we proved a roundedness criterion based on the constancy along the boundary of a $C ^2$ domain of a certain function, depending on the normal distance and on the principal curvatures, see \cite[Thm.~1]{CFa}. If compared to such result, the roundedness criterion stated in  Corollary \ref{cor:lambda} of the present paper has the advantages of applying to any $C ^1$ domain, and  of involving uniquely the normal distance; moreover, it is obtained through completely different techniques, of more geometrical nature. 

\smallskip
We conclude by observing that clearly questions (*) and (**) can be raised also in space dimensions higher than $2$ (or even in a Riemannian manifold), but they seem much more difficult to solve.  Nonetheless, concerning question (**), we are able to deal with domains in the $n$-dimensional Euclidean space, under the severe restriction that they are convex (see Theorem \ref{t:nconvex}). Removing this restriction remains by now an open problem. 

\smallskip
We defer to a companion paper \cite{CFc} some applications of the geometric results contained in this manuscript to PDEs, specifically to boundary value problems involving the infinity-Laplacian operator. 

\smallskip
The outline of the paper is the following: hereafter we fix some notation; in Section \ref{main} we state the main results;  
in Section \ref{secnot} we provide some background material; Section \ref{secanal} is devoted to some intermediate key results, which prepare the proofs given in Section \ref{secproofs}.

\bigskip
{\maius Acknowledgments.}
The authors would like to thank Piermarco Cannarsa
for pointing out the paper \cite{ACKS}.

\bigskip
{\maius Notation.}  
The standard scalar product of two vectors $x,y\in\R^n$
is denoted by $\pscal{x}{y}$, and $|x|$ stands for the
Euclidean norm of $x\in\R^n$.
Given an open bounded domain $\Omega\subset \R ^n$, we denote
by $|\Omega|$ and $|\partial\Omega|$ respectively its $n$-dimensional Lebesgue measure 
and the $(n-1)$-dimensional Hausdorff measure of its boundary. We set $\Omega ^c:=\R^ n \setminus \Omega$.

We call $B _r (p)$ the open disk of center $p$ and radius $r$, and $\overline B _r (p)$ its closure. 
We indicate by $[p,q]$ the line segment with extremes $p$ and $q$. 

As customary, we say that a function is of class $C^{k, \alpha}$ when all its derivatives up to order $k$ satisfy a H\"older condition of exponent $\alpha$, and that it is of class $C ^ \omega$ when it is analytic.

By saying that an open set $\Omega\subset\R^n$
(or, equivalently, its closure
$\overline{\Omega}$ or its boundary $\partial \Omega$)
is of class $C^k$, $k\in\N$,
we mean that, for every point $x_0\in\partial \Omega$
there exists a neighborhood $U$ of $x_0$ and a bijective
map $\psi\colon B_1(0)\to U$ such that
$\psi\in C^k(B_1(0))$, $\psi^{-1}\in C^k(U)$,
$\psi(B_1(0)\cap \{x_n>0\}) = \Omega\cap U$,
$\psi(B_1(0)\cap\{x_n=0\}) = \partial\Omega\cap U$.
An analogous definition holds with $C^{k,\alpha}$, $C^{\infty}$,
$C^{\omega}$ instead of $C^k$.

Given a closed set $S\subset \R^n$, we denote by $d_S$ the distance function from $S$, 
defined by
\[
d_{S}(x) := \min_{y\in S} |x-y|,\quad x\in\overline{\Omega}\ ,
\]
where $|\cdot |$ is the Euclidean norm in $\R ^n$, and by 
$\pi _S$ the projection map onto $S$, namely, for every $x \in \R ^n$, we call $\pi_{S} (x)$ the set of points $y\in S$ such that 
\[
|x - y| = d_{S} (x)  \, .
\]
Whenever $x$ has a unique projection onto $S$,
with a minor abuse of notation we shall identify 
the set $\pi_S(x)$ with its unique element.
 
Moreover, for $r>0$ we denote by $S_r$ the $r$-tubular neighborhood of $S$: 
\[
S_r := \big \{ x \in \R^n \ :\ d_S (x) < r \big \}\ .
\]

\section{Main results}\label{main}

\begin{definition} 
\label{d:psm}
We say that a set $S \subset \R ^n$ is {\it proximally $C ^k$ (of radius $r_S$)} if it is nonempty, closed, and there exists $r_S>0$ such that
the distance function $d_S$ is of class $C ^ k$ in the set $\{ x \in \R^n \ :\ 0 < d _S (x) < r_S \}$. 
\end{definition}

Notice that proximally $C^1$ sets according to the above definition correspond to sets which in the literature are usually named  proximally smooth, or with positive reach, as discussed in the Introduction. 

\medskip
Our main results are the following characterizations of planar sets $S$ which satisfy one of the following conditions:

\medskip
\begin{itemize}
\item[(H1)]
$S\subset\R^2$ is connected, with empty interior, proximally $C ^1$; 

\medskip
\item[(H2)]
$S\subset\R^2$ is connected, with empty interior, proximally $C ^2$. 
\end{itemize}

\begin{theorem}
\label{t:c11}
Assume that $S\subset \R ^2$ satisfies $(H1)$. 

Then  $S$ is either a singleton, or
a $1$-dimensional manifold
of class $C^{1,1}$.
\end{theorem}

\begin{theorem}
\label{t:c2}
Assume that $S\subset \R ^2$ satisfies $(H2)$.

Then $S$ is either a singleton, or a
$1$-dimensional manifold without boundary of class $C^2$. 
\end{theorem}

\smallskip
\begin{remark}\label{r:comments}
(i) Clearly,  if the assumption $S$ connected is removed from $(H1)$ and $(H2)$, Theorems \ref{t:c11} and \ref{t:c2} can be applied to characterize each connected component of $S$. 

\smallskip
(ii) If the assumption $S$ bounded is added to $(H2)$, Theorem \ref{t:c2} allows to conclude that $S$ is a regular simple closed curve of class $C ^2$.

\smallskip
(iii) If the regularity requirement in condition $(H2)$ is strengthened by asking that $S$ satisfies Definition \ref{d:psm} with $C ^2$ replaced either by $C ^ {k, \alpha}$, for some $k \geq 2$ and $\alpha \in [0,1]$, 
or by $C ^ \infty$, or by $C ^ \omega$, 
then the thesis of Theorem \ref{t:c2} can be strengthened accordingly, namely the manifold $S$ turns out to be respectively of class $C ^ {k, \alpha}$, $C ^ \infty$,  or $C ^ \omega$ (cf.\ Remark \ref{r:inspection}). 

\smallskip
(iv) It is a natural question to ask whether Theorem \ref{t:c11} still holds if the condition $S$ proximally $C ^1$ is weakened into an {\it exterior sphere condition}. Namely, if $S$ is proximally $C^1$ of radius $r_S$, for every $r \in (0, r_S)$, every $x\in S$ and every unit vector $\zeta$ such that $x \in \pi_S (x+r\zeta)$, the ball of radius $r$ centered at $x + r \zeta$ does not intersect $S$ (see e.g. \cite[Thm.~4.1 (d)]{CSW}).  
At least without any additional assumption on $S$, the converse implication is not true:
the exterior sphere condition  is strictly weaker than proximal smoothness (see \cite{NoStTa}), and it turns out that it is not sufficient to guarantee the validity of Theorem \ref{t:c11}. Examples of sets which satisfy an exterior sphere condition but are not a manifold of class $C^{1,1}$, or not a manifold at all, can be easily constructed: think for instance to the graph of the function $|x|$, or to the union of two mutually tangent circumferences. 
\end{remark}

We now turn attention to the consequences of Theorems \ref{t:c11} and \ref{t:c2} on the geometry of planar domains
whose high ridge and cut locus coincide. 
We are going to see that such domains admit a simple geometrical characterization, as tubular neighborhoods of a $C ^ {1,1}$ manifold; moreover
such characterization turns into a symmetry statement in case the involved domain is $C ^2$ and simply connected. 

In order to state these results
more precisely, and since the terminology adopted in this respect in the literature is not uniform,
let us fix some notation concerning the geometry of the distance function from the boundary.

\begin{definition}\label{maximal} 
Let $\Omega \subset \R^n$ be an open bounded domain. 

\medskip
--  $\Sigma (\Omega)$:= the {\it skeleton} of $\Omega$ is the singular set of $d _{\partial \Omega}$ (i.e., the set of points $x \in \Omega$ such that $d _{\partial \Omega}$ is not differentiable at $x$, or equivalently such that $\pi _{\partial \Omega}(x)$ is not a singleton);

\medskip
-- $\Cut(\Omega)$:= the {\it cut locus} of $\Omega$ is the closure of $\Sigma (\Omega)$ in $\overline \Omega$;

\medskip

-- $\cent(\Omega)$:= the {\it central set} of $\Omega$ is the set
of the centers of all maximal balls contained into $\Omega$. (We  say that an open ball $B_r (p)$ is a {\it maximal ball} contained into $\Omega$ if 
$B_r (p) \subset \Omega$ and there does not exist any other open ball strictly containing $B_r (p)$ which is still contained into $\Omega$.) 

\medskip

-- $\high (\Omega)$:= the {\it high ridge} of $\Omega$ is the set where $d _{\partial \Omega}$ attains 
its maximum over $\overline \Omega$ . 
\end{definition}

\bigskip
Several topological and structure properties of these sets are known; some of them, which will be needed somewhere in the paper, are recalled in Section \ref{secnot} (see Proposition \ref{top}).
Here let us just recall that, for a general domain $\Omega$, there holds
\begin{equation}\label{f:inc}
\high(\Omega) \subseteq \Sigma (\Omega ) \subseteq\cent(\Omega) \subseteq \Cut(\Omega).
\end{equation}

Indeed, the inclusion $\high (\Omega) \subseteq \Sigma (\Omega)$  follows immediately from the eikonal equation; for the remaining inclusions see \cite[Thm.~3B]{Frem}. 

We point out that these inclusions may be strict. Simple examples are the following: when $\Omega = R$ is a rectangle one has
\[
\high(R) \subsetneq \Sigma (R ) =  \cent(R) \subsetneq \Cut(R),
\]
while $\Omega = E$ is an ellipse one has
\[
\high(E) \subsetneq \Sigma (E ) \subsetneq\cent(E) =  \Cut(E). 
\]
More pathological examples, where these sets turn out to be ``substantially'' different, are indicated in Remark \ref{r:pat} below. 

We now turn our attention to the question stated as (**) in the Introduction: what can be said about planar domains $\Omega$ for which all the inclusions in (\ref{f:inc}) become equalities? The answer is contained in the next statement.

\begin{theorem} \label{forBK}
Let $\Omega \subset \R ^2$ be a nonempty open bounded  connected domain such that 
\begin{equation}\label{f:ug}
\high(\Omega) = \Cut (\Omega)=:S\,.
\end{equation}
  
Then $S$ is either a singleton or a $1$-dimensional manifold of class $C ^ {1,1}$ and, setting $\rho_\Omega := \max _{\overline \Omega} d _{\partial \Omega}$,  $\Omega$ is the $\rho_\Omega$-tubular neighborhood
$$\Omega=  S_{\rho_\Omega} := \{ x \in \R ^2 \ :\ d_S (x) < \rho_\Omega \}\,.$$

In particular, if $\Omega$ is $C ^2$, then $S$ is either a singleton or a
$1$-dimensional manifold without boundary of class $C^2$,  and $\Omega = S _{\rho _\Omega}$. 

Finally, if $\Omega$ is also simply connected, then $S$ is a singleton, and 
$\Omega$ is the disk with center $S$ and  radius $\rho _\Omega$.
\end{theorem}

\begin{remark}\label{r:nonreg}
By inspection of the proof of Theorem \ref{forBK}, it follows that, for every $r \in (0, \rho _\Omega)$, the parallel set 
\[
S_r := \big \{ x \in \R ^2 \ :\ d _S (x) < r \big \}  
\]
is of class $C^{1,1}$. We point out that this is not necessarily true also for $r = \rho _\Omega$. 
In other words, a domain $\Omega$ satisfying the assumptions of Theorem \ref{forBK} does not need to be of class $C^{1,1}$, nor $C ^1$. 
For instance, let $p:= (-1, 1)$, $q:= (0,1)$, $a:= (1,0)$, $b: = (1, -1)$, and define $S$ by
$$S:= [p,q] \cup \big \{ \partial B_1 (0) \cap \{ x_1\geq 0, x_2 \geq 0 \}  \big \} \cup [a,b]\,.$$
Then the $1$-tubular neighbourhood of $S$, namely $\Omega=\big \{ x \in \R ^2 \, :\,  d _S (x) < 1 \big \}$ satisfies the assumptions of Theorem \ref{forBK}, and in particular condition (\ref{f:ug}), but is not of class $C ^1$ 
(see Figure~\ref{fig:nonreg} left).
\end{remark}

\begin{figure}[ht]
\begin{minipage}{0.5\linewidth}
\centering
\includegraphics[height=3cm]{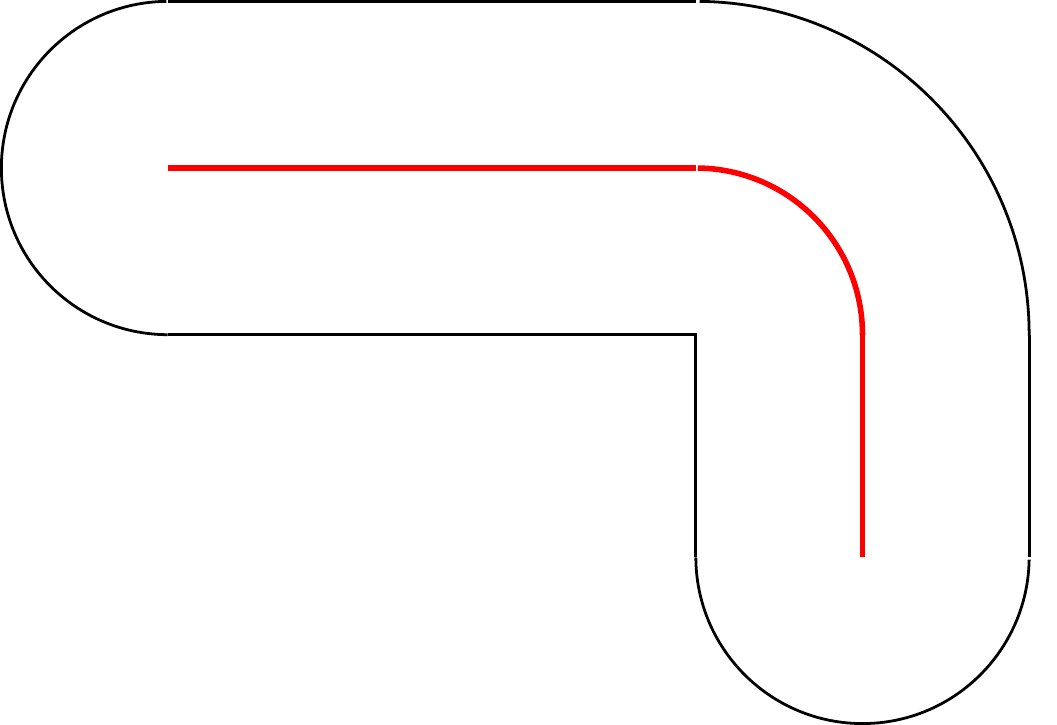}
\end{minipage}%
\begin{minipage}{0.5\linewidth}
\centering
\includegraphics[height=3cm]{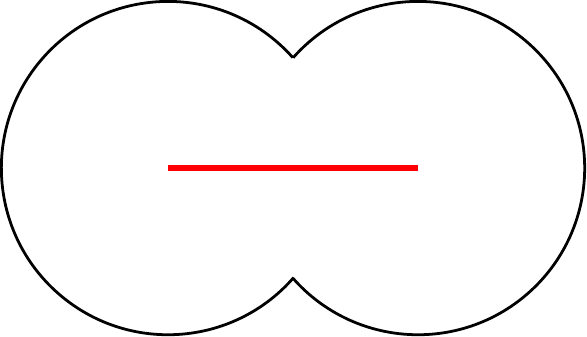}
\end{minipage}
\caption{The sets described in Remarks \ref{r:nonreg} and \ref{r:lambda}.}   
\label{fig:nonreg}   
\end{figure}

\begin{remark}\label{r:MAT}
Using the notation of \cite{Degen2004},
a maximal disk $D$ in $\Omega$ is said to be
\textsl{regular} if the contact set
$\partial D \cap \partial\Omega$ contains exactly two points,
and \textsl{singular} if this is not the case.
Then, if $\Omega$ satisfies the assumptions of 
Theorem~\ref{forBK},
and denoting by $S^*$ the (possibly empty) boundary of
the manifold $S$,
we have that all maximal disks centered at
$S\setminus S^*$ are regular,
while the ($0$ or $2$) maximal disks centered at $S^*$
are singular. 
\end{remark}

\bigskip

Let us now restrict attention to domains $\Omega$ of class $C ^1$.  
For such a domain, let $\nu_{\Omega} $ denote the inner unit normal to $\partial \Omega$, and 
%
let us recall the following definition of normal distance: 

\begin{definition}
\label{cut} 
Let $\Omega \subset \R^n$ be an open bounded domain of class $C^1$. For every $y \in \partial \Omega$,  its {\it normal distance to the cut locus} is given by  
$$\lambda _\Omega(y) := \sup \big \{ t \geq 0 \ :\ \pi_{\partial \Omega} (y + t \nu_\Omega (y) )  = \{ y \} \big \}\ ,$$
\end{definition}

As a consequence of Theorem \ref{forBK}, we are able to characterize planar domains $\Omega$ of class $C ^1$ with constant normal distance along the boundary: 

\begin{corollary}\label{cor:lambda}
Let $\Omega \subset \R ^2$ be an open bounded  connected domain of class $C ^1$ such that, for all $y \in \partial \Omega$,
\begin{equation}\label{f:lambdac} 
\lambda _\Omega(y) = \text{constant}\,.
\end{equation}
Then $\Omega$ satisfies $(\ref{f:ug})$ and hence its geometry can be characterized according to Theorem \ref{forBK}. 
\end{corollary}

\begin{remark}\label{r:lambda}
We point out that the assumption $\Omega \in C ^1$ in Corollary \ref{cor:lambda} cannot be weakened. To be more precise notice first that, if $\Omega$ is just piecewise $C ^1$, Definition \ref{cut}  of the function $\lambda _\Omega$ can still be given for $y$ belonging to $\partial \Omega$ except a finite number of points (those where $\nu _\Omega$ is not defined). 
Nevertheless,  if equality (\ref{f:lambdac}) is valid only ${\mathcal H} ^ {1}$-a.e.\ on $\partial \Omega$, the geometric condition $\high (\Omega) = \Cut (\Omega)$  is not necessarily true. For instance, let $\alpha >0$, let $r \in (\alpha, 2 \alpha)$, and  let $\Omega:=  B _ r (p) \cup B _ r (q)$, where $p := ( - \alpha, 0)$, and $q:= (\alpha, 0 )$
(see Figure~\ref{fig:nonreg} right). 
Then we have
$\lambda _\Omega(y) = r$ for all $y \in \partial \Omega \setminus \big ( \partial B _ r (p) \cap \partial B _ r (q))$, but
$\{p\} \cup \{q\} = \high (\Omega) \subsetneq \Cut (\Omega) = [p,q]$.
\end{remark}

Extending the above results to higher dimensions seems to be a delicate task. So far, we have the following generalization
of Theorem \ref{forBK}, which settles the case  of  {\it convex} sets in $n$ dimensions:

\begin{theorem}\label{t:nconvex}
Let $\Omega\subset\R^n$  be a nonempty open bounded convex set
of class $C^2$, satisfying $(\ref{f:ug})$. Then $S$ is a singleton and $\Omega$ is a ball.
\end{theorem}

%
%
%

\section {Background material} \label{secnot}

In order to be as possible self-contained, in this section we give a quick overview of some properties
of proximally smooth sets (cf. Proposition \ref{l:prox}) and of the sets introduced in Definition \ref{maximal} (cf. Proposition \ref{top}), which will be needed at some point in the paper. 

\begin{proposition}\label{l:prox} 

Let $S \subset \R ^n$ be proximally $C ^1$ of radius $r_S$, and let $r \in (0, r_S)$. Then:
\begin{itemize}
\item[(i)] on the set $\{ 0 < d_S(x) < r_S\}$, the Fr\'echet differential of $d_S$ is given by 
$$d' _S(x)  = \frac{ x - \pi _ S(x)} {d_S(x)}\, ;$$

\smallskip
\item[(ii)] on the set $\{ 0 < d_S(x) < r\}$, the projection map $\pi _S$ is Lipschitz of constant $\frac{r_S} { r_S - r}$; 
in particular, 
the map $d_S$ is of class $C^{1, 1}_{\text{loc}}$ on the set $\{ 0 < d_S(x) < r_S\}$;

\smallskip
\item[(iii)]   the following equalities hold:
\begin{eqnarray}  d_{(S_r) ^c} (x)  = r-  d _ S(x)  \qquad \hbox{ on } \{ 0 < d_S(x) < r\}\, ,
 \label{ugdist} \\ \noalign{\smallskip}
 d_{\overline{S_r} } (x) = d _ S(x) -r  \qquad \hbox{ on } \{ r < d_S(x) < r_S\}
 \label{ugdist_bis}
 \, , 
\end{eqnarray}
implying in particular that $\overline {S_r}$ is proximally smooth of radius $r_S - r$;


\smallskip
\item[(iv)]   the set $S_r$ is of class $C ^{1,1}$;
\smallskip


\smallskip
\item[(v)] if in addition 
$S$ satisfies Definition \ref{d:psm} with $C ^1$ replaced either by $C ^ {k, \alpha}$ (for some $k \geq 2$ and $\alpha \in [0,1]$), 
or by $C ^ \infty$, or by $C ^ \omega$, then 
the set $S_r$ is respectively of class $C^{k, \alpha}$, $C ^ \infty$, or $C ^ \omega$.  
\end{itemize}
\end{proposition}

\begin{proof}
We refer to \cite{CSW}: for (i), see Thm.~ 3.1; for (ii), see Thm.~4.8; for (\ref{ugdist}), see Thm.~ 4.1 (c); for (\ref{ugdist_bis}), see Lemma 3.3; for (iv), see Corollary 4.15  and use also the $C ^ {1,1}_{\rm loc}$ regularity of $d_S$ stated at item (ii).
Finally, (v) can be easily obtained as follows: if $d_S$ is of class $C^{k, \alpha}$, $C ^ \infty$, or $C ^ \omega$  on the set $\{ 0 < d_S(x)< r _S \}$, since on the same set by (i) it holds $\|d'_S (x) \| = 1$,  by the Implicit Function Theorem $S_r$ inherits the same regularity. 
\end{proof}


\medskip
\begin{proposition} \label{top}

Let $\Omega \subset \R^n$ be an open bounded domain. 
\begin{itemize}

\item[(i)] $\Sigma (\Omega)$ is ${ C} ^2$-rectifiable, namely it can be covered up to a ${\mathcal H} ^ {n-1}$-negligible set by a countable union of embedded $(n-1)$-manifolds of class $C ^2$; in particular, $\Sigma (\Omega)$ has null Lebesgue measure. 

\item[(ii)] $M (\Omega)$ has null Lebesgue measure. 

\item[(iii)]  $\Sigma (\Omega)$ has the same homotopy type as $\Omega$.

\item[(iv)] If $\Omega \in C^ 2$,  it holds
$$\cent(\Omega) 
= \Cut(\Omega)\,.$$ 
Moreover, in this case $\Cut(\Omega)$ has null Lebesgue measure, is contained into $\Omega$, and
$d_{\partial \Omega}$ is of class $C^2$ in 
$\overline \Omega \setminus \Cut(\Omega)$.  
\end{itemize}
\end{proposition}

\begin{proof} 
(i) The fact that $\Sigma (\Omega)$ has null Lebesgue measure follows from Rademacher Theorem. 
Since $d_{\partial\Omega}$ is locally semiconcave in $\Omega$,
the $C^{2}$-rectifiability of $\Sigma(\Omega)$ 
follows from the structure result proved in \cite{Alb}. 

(ii) See \cite[Prop.~3N]{Frem}. 

(iii) See \cite[Thm.~6]{ACKS}, \cite[Thm.~4.19]{Lie}.  

(iv)  See  \cite[Sect.~6]{CMf}. 
\end{proof}

\medskip
\begin{remark}\label{r:pat}
We remark that the property of $\Sigma(\Omega)$ and $M (\Omega)$ of having null Lebesgue measure is not enjoyed in general by $\Cut (\Omega)$:   in  \cite[Section 3]{MM}, there is an example  of two-dimensional convex set $\Omega$ whose cut locus has positive Lebesgue measure.  We also point out that the central set $\cent (\Omega)$ of a planar domain may fail to be ${\mathcal H} ^ {1}$-rectifiable (see the examples in \cite[Section 4]{Frem}), and it may even happen to have Hausdorff dimension $2$ (see \cite{BiHa}). 
\end{remark}

\section{Analysis of the contact set}\label{secanal}
\bigskip

Throughout this section, we work in two space dimensions. 
We start by elucidating the geometry of tubular neighborhoods of a set which satisfies (H1): 

\begin{lemma} Let $S \subset \R ^2$ satisfy $(H1)$, and  let $r$ be a fixed radius in $(0, r_S)$. 
Then it holds
\begin{eqnarray} S = \high ( S_r) = \Sigma (S_r) = \cent(S_r) = \Cut(S_r) \label{f:esse}
 \\ \noalign{\smallskip}
\lambda _{S_r} (y) = r \qquad \forall y \in \partial S _r\,. \label{f:lSr}
\end{eqnarray}
\end{lemma} 

\proof We observe that
\begin{equation}\label{ugdist_ter} d_{(S_r) ^c} (x) = r -  d _ S(x)  \qquad \forall x \in S _ r  \, . 
\end{equation}
Indeed, for $ x\in S _r \setminus S$, the above equality holds true by (\ref{ugdist}) in Proposition \ref{l:prox} (iii). On the other hand, since by assumption $S$ has empty interior, its complement $S^c$ is dense in $\R ^2$. Then, given $x \in S$, there exists a sequence $\{ x _h \}$ contained into $S _r \setminus S$, with $\lim _h x _h = x$. By applying (\ref{ugdist}) to each $x_h$, and then passing to the limit as $h \to + \infty$, we get $d_{(S_r) ^c} (x) = r$, which extends the validity of (\ref{ugdist}) to $S$ and proves (\ref{ugdist_ter}). In view of  (\ref{ugdist_ter}), it is clear that $S = \high (S_r) = \cent (S_r)$;
then (\ref{f:esse}) follows recalling (\ref{f:inc}) and the fact that $S$ is closed. 
After noticing that $\lambda _{S_r}$ is well-defined thanks to Proposition \ref{l:prox} (iv), equality (\ref{f:lSr}) readily follows from Definition \ref{cut} and (\ref{f:esse}). 
\qed

\bigskip
\begin{definition}
Let $S \subset \R ^2$ satisfy $(H1)$, let $p \in S$, and let $r$ be a fixed radius in $(0, r_S)$.
We call {\it contact set} of $p$ into $S_r$ the  intersection of $\partial S_r$ and the closure of $B _ r (p)$ 
 (which is a maximal disk contained into $S_r$): 
$$
C_r(p) := \partial B _r (p) \cap \partial S_r = \big \{ y \in \partial S_r \ :\ |y - p| = r \big \}\ , \qquad p \in S\, . 
$$
\end{definition}

\bigskip
\begin{remark}\label{r:two} By its definition, $C_r (p)$ is a nonempty closed set, whose 
connected components are singletons or closed arcs. Moreover, in view of (\ref{f:esse}), $C_r (p)$ contains at least two points 
(see \cite[Corollary~1, p.~67]{Clar}). Notice also that, since $r < r_S$, it holds
$C_r(p) \cap C _r(q) =  \emptyset$ if  $p \neq q$.  \end{remark}

We are now going to carry on a thorough geometric analysis of the contact set $C_r (p)$: our objective is giving a complete characterization of it, which will be achieved in Proposition \ref{p:reg}. 
As intermediate steps, in the following two lemmas  we begin the investigation of the singletons and the arcs which form $C_r (p)$. 

\begin{lemma}
\label{l:ang}
Let $S\subset \R^2$ satisfy $(H1)$ and let $r \in (0, r_S)$.
Let $p\in S$, and let $a,b \in C_r(p)$.  
If $a$ and $b$ are distinct and not antipodal, then $C_r(p)$
contains the arc of $\partial B _r (p)$ of length $< r\pi$
joining $a$ and $b$.
\end{lemma}

\begin{proof}

Consider the cone
\[
\Sigma_+:= p + \{\alpha  (a-p) + \beta (b-p) \, : \, \alpha , \beta \geq 0\}.
\]
We have to prove that 
\[
\left[\partial B_r(p) \cap \Sigma_+\right]  \subseteq C_r(p)\ .
\]

We claim that there exists $\delta \in (0, r)$ such that
\begin{equation}
\label{f:pd}
\pi _ S \big ( \partial S_r \cap \Sigma_+ \cap B _ \delta (a) \big ) \subseteq \{p \}\ .
\end{equation}

Since the vectors $a-p$ and $b-p$ are not parallel, we have 
\[
\epsilon := \left| \frac{a+b}{2} - p\right| > 0.
\]
By the definition of $\epsilon$, we have
$$\big [ \big (B_{\epsilon}(p)\setminus \{p \} \big ) \cap \Sigma_+ \big ] \subset \big [{B}_r(a)\cup{B}_r(b)\big ]\ .$$
Recalling that by construction $B _r (a)$ and $B _r (b)$ cannot intersect $S$, we infer that
\begin{equation}\label{e1}
\big [ B_{\epsilon}(p) \cap \Sigma_+ \cap S \big ] = \{ p \}\ .
\end{equation}

Now we recall that the projection map $\pi_S$ is Lipschitz continuous on $\partial S_r$ 
with constant $C := r_S/(r_S -r)$ ({\it cf.}\ Proposition \ref{l:prox} (ii)). 
Therefore, 
if we choose $\delta := \epsilon /C$
we get
\begin{equation}\label{e2}
\pi _ S \big ( \partial S_r \cap \Sigma_+ \cap B _\delta (a)  ) \subset B _\epsilon (p)\,.
\end{equation} 
By (\ref{e1}) and (\ref{e2}) 
we conclude that \eqref{f:pd} holds, proving the claim.

\smallskip
Since $\partial S_r$ is tangent to $\partial B_r(p)$ at $a$,
it is not restrictive to assume that the arc-length
parametrization $\gamma\colon [0, L]\to\mathbb{R}^2$
of the connected component of $S_r$ containing $a$ satisfies
$\gamma(0) = a$ and $\gamma(s) \in \Sigma_+$ for
$s\geq 0$ small enough.
Let
\[
\bar{s} := \sup\left\{
s>0:\ \gamma([0,s])\subset\Sigma_+
\right\}\,.
\]
Clearly we have $0 < \bar{s} < L$.
Let $s_1 >0$ be such that $\gamma(s)\in B_{\delta}(a)\cap\Sigma_+$
for every $s\in [0, s_1]$.
From \eqref{f:pd} we deduce that
\[
\pi_S(\gamma(s)) = \{p\}\quad
\forall s\in [0, s_1]\,,
\]
hence the restriction of $\gamma$ to $[0,s_1]$
parametrizes an arc of length $r\, s_1$
on $\partial B_r(p)$
joining $a$ to $\gamma(s_1)$.
Thus, if $s_1 = \bar{s}$,
then $\gamma(s_1) = b$ and we are done.

Otherwise, denoting by $K$ the $L ^ \infty$-norm of
the curvature of $\gamma$ (which only depends on $r$ and $r_S$, again thanks to Proposition \ref{l:prox} (ii)), we observe that 
we can choose
\[
s_1 \geq \min\left\{\delta,\, \frac{\pi}{K} \right\}\ ,
\]
as  $\delta$ is the shortest possible exit-time from $B _\delta (a)$ and $\frac{\pi}{K}$ the shortest possible exit-time from $\Sigma_+$. 
{}

Hence, we can repeat the same argument replacing the point $a$ by 
$a' = \gamma(s_1)$, after noticing that 
\[
\left| \frac{a'+b}{2} - p\right| > \epsilon
\]
and so \eqref{f:pd} holds with $a$ replaced by $a'$ and the
same value of $\delta$.

In a finite number of steps we can construct numbers
$0 = s_0 < s_1 < \ldots < s_N = \bar{s}$ with
\[
s_{j} - s_{j-1} \geq
\min\left\{\delta,\, \frac{ \pi}{K}\right\},
\quad
\forall j = 1, \ldots, N-1
\]
such that 
the restriction of $\gamma$ to $[s_{j-1},s_j]$
is a parametrization of an arc of length $r\, (s_j - s_{j-1})$
on $\partial B_r(p)$
joining $\gamma(s_{j-1})$ to $\gamma(s_j)$,
and $\gamma(s_N) = \gamma(\bar{s}) = y$,
completing the proof.
\end{proof}

\begin{lemma}
\label{l:conv}
Let $S \subset \R^2$ satisfy $(H1)$ and let $r \in (0, r_S)$.
Let $p\in S$, and assume that $S \neq \{p\}$. If $C _r(p)$ contains a nontrivial arc, 
then $C_r(p)$ is a  connected arc of length $\leq \pi r$.  
\end{lemma}

\begin{proof}
We first prove, arguing by contradiction, that $C_r(p)$ consists of only one connected component.
Let $\Gamma$ be the connected component of $C_r(p)$ containing
the nontrivial arc (so that $\Gamma$ itself is a nontrivial arc),
and let 
$a\in C_r(p)\setminus\Gamma$ be a point lying in
another connected component of $C_r(p)$.
Clearly, there is at least one endpoint $b$ of $\Gamma$
such that $a$ and $b$ are not antipodal, so that
by Lemma~\ref{l:ang} we get the contradiction. 

It remains to prove that, if $S\neq \{p\}$, then the
length of $\Gamma$ is $\leq \pi r$.
Namely, if this is not the case, by Lemma \ref{l:ang} it turns out that $C_r(p)$ contains also $\partial B _ r (p ) \setminus \Gamma$. 
Thus $C_r (p) $ contains the whole circumference $\partial B _ r (p)$. 
Since $S$ is connected, this means that $S = \{p\}$, against the assumption. 
\end{proof}

\bigskip
We are now ready to give the complete picture of $C_r (p)$: 

\begin{proposition}
\label{p:reg} Let $S \subset \R^2$ satisfy $(H1)$ and let $r \in (0, r_S)$.
Let $p\in S$, and assume that $S \neq \{p\}$.
Then $C_r(p)$ consists either of only two antipodal points, or of a closed semicircumference. 
\end{proposition}

\proof
By Remark~\ref{r:two}, 
we know that $C_r(p)$ contains at least two points.
Assume that $C_r(p)$ does not contain only two antipodal points.
Then, by Lemma~\ref{l:ang},  $C_r(p)$ contains a nontrivial arc. 
In turn, by Lemma~\ref{l:conv}, this implies that $C_r(p)$ is a connected arc of length $\leq \pi r$.  
We have to show that such arc is precisely a semicircumference.

We argue by contradiction:
let $a,b$ be the endpoints of  $C_r(p)$ and assume by contradiction 
that $a$ and $b$ are not antipodal.  Then, there exists
$\theta_0\in (0, \pi/2)$ so that
the angle in $(0, \pi)$ formed by $a-p$ and $b-p$
is $\pi - 2\theta_0$.
We first prove the following

\medskip
{\it Claim: There exist two cones $\Sigma_a$ and $\Sigma _b$, with vertex in $p$, axis orthogonal to $a-p$ and $b-p$ respectively, 
direction such that $\Sigma _ a \cap C_r(p) = \Sigma _ b \cap C_r (p ) = \emptyset$, 
and half-width $\epsilon < \min \{ \theta _0, \frac{\pi}{2} - \theta _0 \}$, such that
both $\Sigma_a$ and $\Sigma _b$ contain a nontrivial arc of $S$ passing through $p$.}

\medskip

\begin{figure}[htbp]
\begin{minipage}{0.5\linewidth}
\centering
\def\svgwidth{7cm}   
\begingroup%
  \makeatletter%
  \providecommand\color[2][]{%
    \errmessage{(Inkscape) Color is used for the text in Inkscape, but the package 'color.sty' is not loaded}%
    \renewcommand\color[2][]{}%
  }%
  \providecommand\transparent[1]{%
    \errmessage{(Inkscape) Transparency is used (non-zero) for the text in Inkscape, but the package 'transparent.sty' is not loaded}%
    \renewcommand\transparent[1]{}%
  }%
  \providecommand\rotatebox[2]{#2}%
  \ifx\svgwidth\undefined%
    \setlength{\unitlength}{377.58676bp}%
    \ifx\svgscale\undefined%
      \relax%
    \else%
      \setlength{\unitlength}{\unitlength * \real{\svgscale}}%
    \fi%
  \else%
    \setlength{\unitlength}{\svgwidth}%
  \fi%
  \global\let\svgwidth\undefined%
  \global\let\svgscale\undefined%
  \makeatother%
  \begin{picture}(1,1.38846134)%
    \put(0,0){\includegraphics[width=\unitlength]{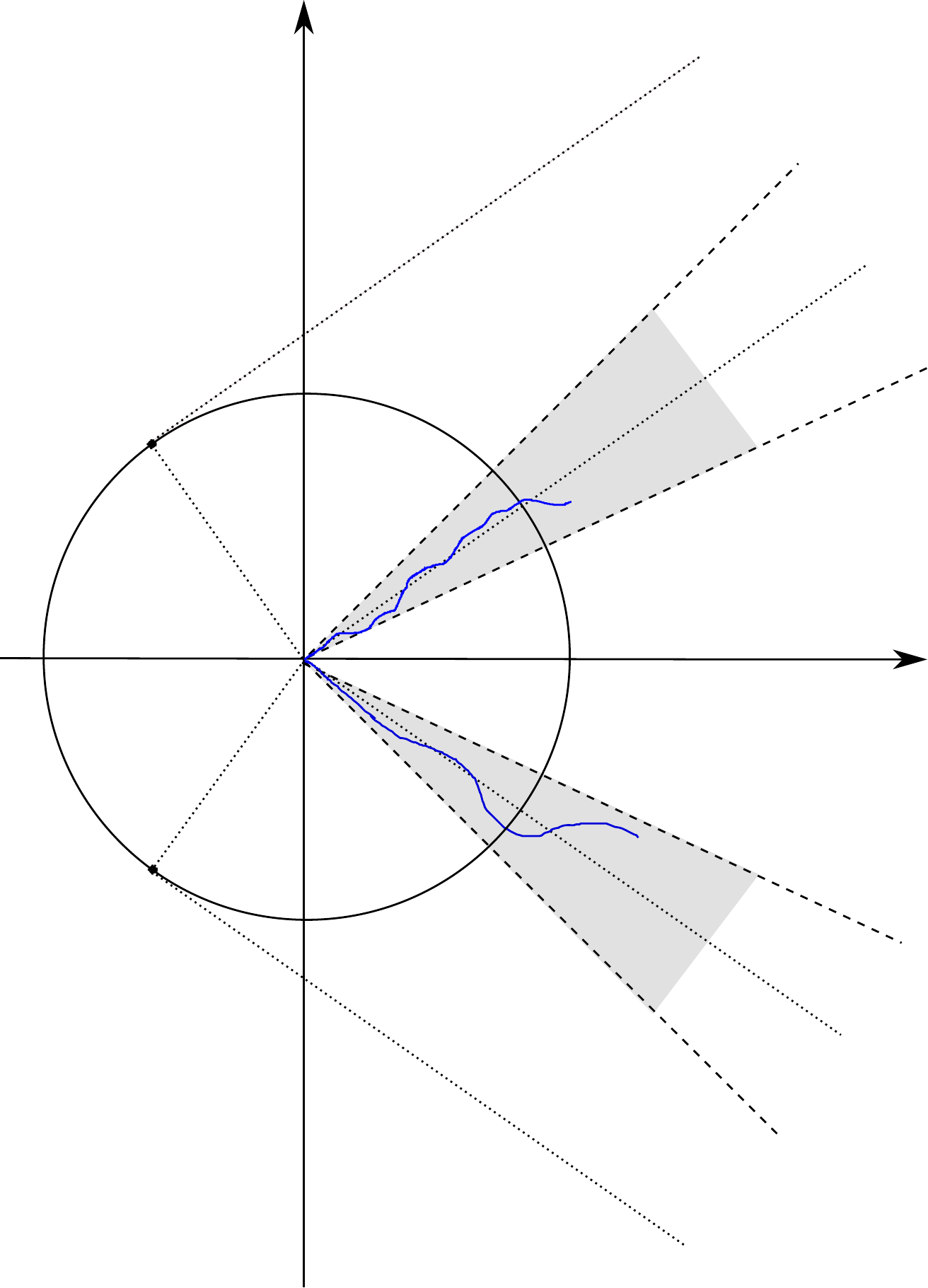}}%
    \put(0.12027376,0.9192436){\color[rgb]{0,0,0}\makebox(0,0)[lb]{\smash{$a$}}}%
    \put(0.11478486,0.40438405){\color[rgb]{0,0,0}\makebox(0,0)[lb]{\smash{$b$}}}%
    \put(0.70539138,0.9280258){\color[rgb]{0,0,0}\makebox(0,0)[lb]{\smash{$\Sigma_a$}}}%
    \put(0.69880452,0.38791724){\color[rgb]{0,0,0}\makebox(0,0)[lb]{\smash{$\Sigma_b$}}}%
  \end{picture}%
\endgroup%
\end{minipage}%
\begin{minipage}{0.5\linewidth}
\centering
\def\svgwidth{7cm}   
\begingroup%
  \makeatletter%
  \providecommand\color[2][]{%
    \errmessage{(Inkscape) Color is used for the text in Inkscape, but the package 'color.sty' is not loaded}%
    \renewcommand\color[2][]{}%
  }%
  \providecommand\transparent[1]{%
    \errmessage{(Inkscape) Transparency is used (non-zero) for the text in Inkscape, but the package 'transparent.sty' is not loaded}%
    \renewcommand\transparent[1]{}%
  }%
  \providecommand\rotatebox[2]{#2}%
  \ifx\svgwidth\undefined%
    \setlength{\unitlength}{379.39037574bp}%
    \ifx\svgscale\undefined%
      \relax%
    \else%
      \setlength{\unitlength}{\unitlength * \real{\svgscale}}%
    \fi%
  \else%
    \setlength{\unitlength}{\svgwidth}%
  \fi%
  \global\let\svgwidth\undefined%
  \global\let\svgscale\undefined%
  \makeatother%
  \begin{picture}(1,1.3818825)%
    \put(0,0){\includegraphics[width=\unitlength]{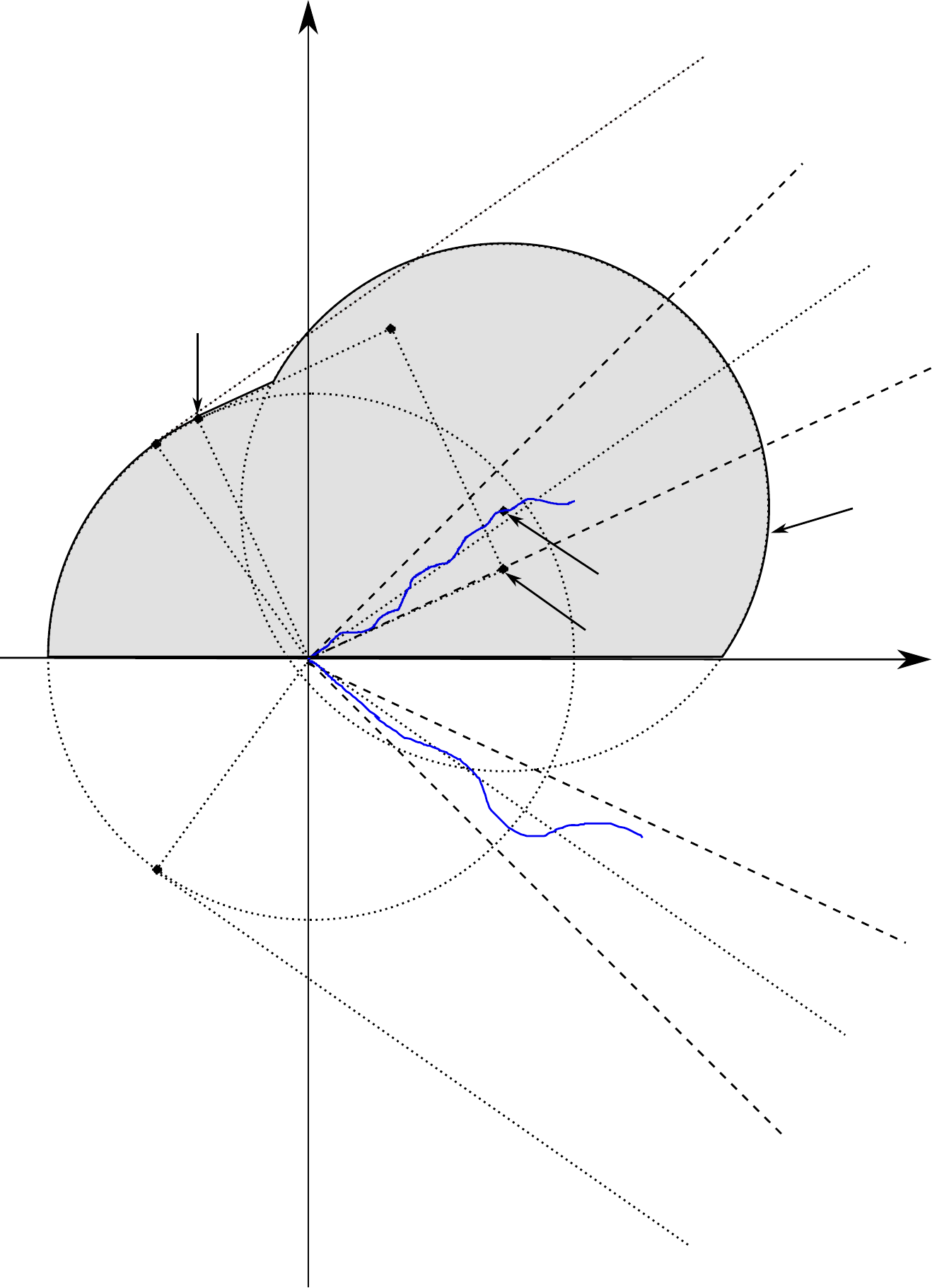}}%
    \put(0.06182972,1.09129171){\color[rgb]{0,0,0}\makebox(0,0)[lt]{\begin{minipage}{0.02191506\unitlength}\raggedright \end{minipage}}}%
    \put(0.03287519,0.32212727){\color[rgb]{0,0,0}\makebox(0,0)[lt]{\begin{minipage}{0.06214957\unitlength}\raggedright \end{minipage}}}%
    \put(0.12446645,0.9149059){\color[rgb]{0,0,0}\makebox(0,0)[lb]{\smash{$a$}}}%
    \put(0.11900365,0.40249398){\color[rgb]{0,0,0}\makebox(0,0)[lb]{\smash{$b$}}}%
    \put(0.65300055,0.75952782){\color[rgb]{0,0,0}\makebox(0,0)[lb]{\smash{$p'$}}}%
    \put(0.170649,1.04229954){\color[rgb]{0,0,0}\makebox(0,0)[lb]{\smash{$p_1$}}}%
    \put(0.63518144,0.69345933){\color[rgb]{0,0,0}\makebox(0,0)[lb]{\smash{$q$}}}%
    \put(0.43761788,1.02773555){\color[rgb]{0,0,0}\makebox(0,0)[lb]{\smash{$q_1$}}}%
    \put(0.87942723,0.85153719){\color[rgb]{0,0,0}\makebox(0,0)[lb]{\smash{$T_+$}}}%
  \end{picture}%
\endgroup%
\end{minipage}
\caption{Proof of Proposition~\ref{p:reg}}   
\label{fig:rectangle}
\end{figure}
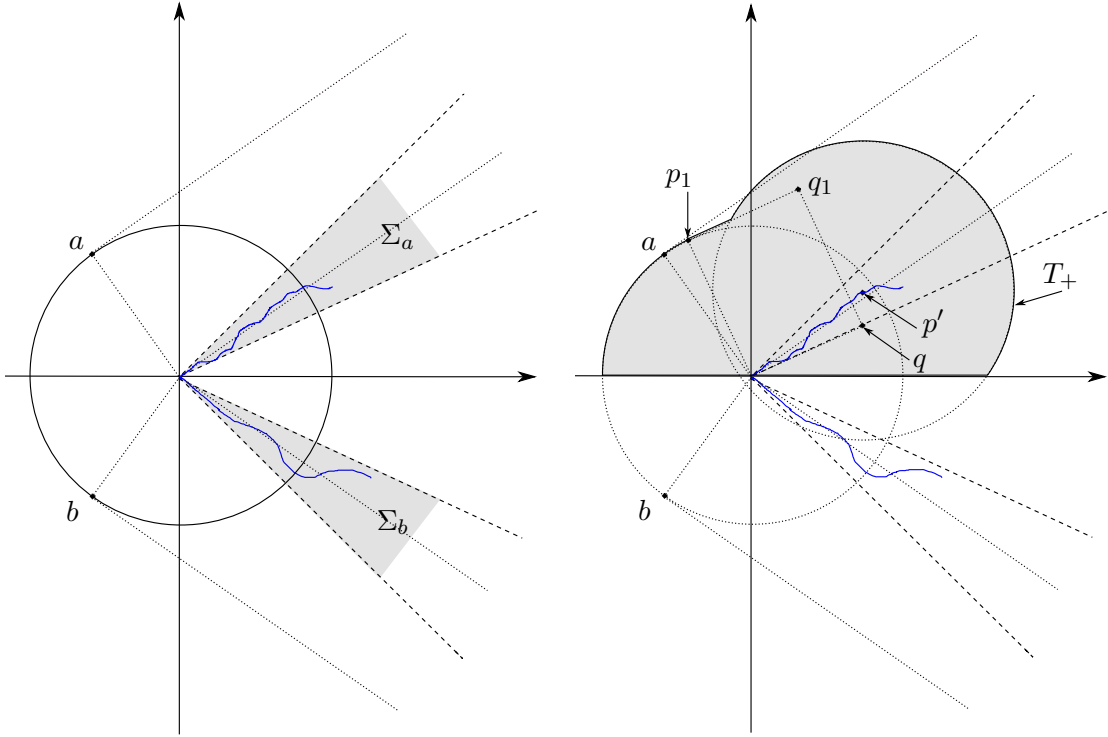

To prove the claim, we can assume
without loss of generality that 
\[
\begin{array}{ll}
& p=(0,0)
\\ \noalign{\medskip}
& a= \big (r\cos (\theta _0 + \frac{\pi}{2} ), r\sin(\theta _0 + \frac{\pi}{2} ) \big ) 
\\ \noalign{\medskip}
& b= \big (r\cos(\theta _0 + \frac{\pi}{2} ), -r\sin(\theta _0 + \frac{\pi}{2} )\big )\, .
\end{array}
\]

We choose $\epsilon <\min\big\{\theta _0, \frac{\pi}{2} - \theta _0 \big\},
$
and we define the cones
\[
\begin{array}{ll}
& \Sigma_a := \{(\rho\cos\theta, \rho\sin\theta):\
\rho\geq 0,\ \theta\in [\theta_0-\epsilon, \theta_0+\epsilon]\},
\\ \noalign{\medskip}
& \Sigma_b := \{(\rho\cos\theta, \rho\sin\theta):\
\rho\geq 0,\ \theta\in [-\theta_0-\epsilon, -\theta_0+\epsilon]\}.
\end{array}
\]
By construction, $\Sigma_a$ and $\Sigma _b$ have vertex in $p$, and  axis orthogonal to $a-p$ and $b-p$; 
moreover, by the choice of the width $\epsilon$, $\Sigma _a$ and $\Sigma _b$ are contained respectively in the first and fourth quadrant, 
and in particular $\Sigma _a \cap \Sigma _b = \{p\}$
(see Figure~\ref{fig:rectangle} left). 

Let us show that $\Sigma _a$ 
contains a nontrivial arc of $S$ 
passing through $p $ (being the proof exactly the same for $\Sigma_b$). 

Let  $\gamma$ be an arc-length parametrization of  
the component of $\partial S_r$ containing $a$, such that $\gamma (0) = a$ and $\gamma'(0) = (\cos\theta_0, \sin\theta_0)$.
Since $a$ is an end-point of $C_r (p)$ and $\gamma$ is continuous, we infer that
there exists $\delta >0$ such that 
$$
\gamma(s) \in (a+\Sigma_a)\setminus\overline{B}_r(p)
\qquad \forall s\in(0, \delta).
$$

By continuity of the projection map $\pi _S$,
this implies 
$$
\pi_S(\gamma(s)) \in \Sigma_a\setminus \{ p \} 
\qquad
\forall s\in(0, \delta).
$$


We conclude that $\pi _ S (\gamma (s))$, for $s \in (0, \delta)$, is a nontrivial arc of $S$ passing through $p$ contained into $\Sigma_a$, and the claim is proved.  

\bigskip
The remaining of the proof is devoted to obtain a contradiction. 
We keep the same coordinates as in the proof of the claim. 
Let $\pi _ S (\gamma (s))$, for $s \in (0, \delta)$ be a nontrivial arc of $S$ passing through $p$ contained into $\Sigma_a$. 
Pick a point in the arc, say 
$$p' = \pi _ S (\gamma (s'))= (x', y')\, , \qquad \hbox{ with } s' \in (0, \delta)\, .$$
Choosing $s'$ sufficiently small, we may assume that $\partial B _r (p)$ and $\partial B _ r (p')$
have two intersection points, one of which lying in the half-plane $\{ y < 0 \}$. 

Set
$$p_ 1 := r\, \Big ( \cos ( \theta _0 - \epsilon  + \frac{\pi}{2} ) , \sin ( \theta _0 - \epsilon  + \frac{\pi}{2} ) \Big )\, ,
\quad q  := (x', \tan (\theta _0 - \epsilon) x' ) \,  , \quad
 q_ 1 := q + p _1\, ,$$
so that the straight line through  $p _1$ and $q _1$ has slope $\theta _0 - \epsilon$ and is tangent to both 
$\partial B _ r (p)$ and $\partial B _ r (q)$, respectively at $p_1$ and $q _1$. 
Denote by $R$ the rectangle with vertices $p$, $p_1$, $q$ and $q_1$. 

Since $B _r (z) \subset S_r$ for every $z \in S$, and since by construction $\pi _ S (\gamma (s)) \subset S \cap \Sigma _a$ for all $s \in (0, s')$, 
we infer that the region 
\[
T _+:= 
\Big \{ (x, y)  \in  \big ( B _ r (p) \cup R \cup B _ r (p' ) \big )\ :\   y \geq 0 \Big \} 
\subset
\bigcup_{s\in [0,s']} B_r\left(
\pi_S(\gamma(s))
\right)
\]
is contained into $S_r$  (see Figure~\ref{fig:rectangle} right).  

By considering a nontrivial arc of $S$ passing through $p$ contained into $\Sigma _b$ and arguing in the same way, we obtain that also the region
$$T _- :=\big  \{ (x, y) \ : (x, -y) \in T _+ \big \}$$
is contained into $S_r$. Hence, the same holds true for the region 
$T:= T _+ \cup T _-$. 

Notice that, by construction (and in particular by the choice of $s'$), the only points $\tilde p\in \partial T$ which realize 
the distance of $p$ from $\partial T$ are those of $C_r (p)$, namely it holds
\begin{equation}\label{distp}
\tilde p \in \partial T\, , \quad | p - \tilde  p | = d_{\partial T} (p)  \ \Leftrightarrow \ \tilde p \in C_r (p) \ .
\end{equation}

\medskip
We now consider the point $p _\lambda := (\lambda, 0)$,   for $\lambda >0$ small. 
Clearly, since $|p _\lambda - p| = \lambda$, as soon as $\lambda < r$ it holds 
\begin{equation}\label{contrad1}
p _\lambda \in S _r\, .
\end{equation} 
On the other hand, by  the inclusion $T \subseteq S_r$, it holds
\[
d_{\partial S_r}(p_{\lambda}) \geq d_{\partial T}(p_{\lambda}) \,.
\]

\medskip
Now, for $\lambda >0$ small, 
\[ 
d_{\partial T}(p_{\lambda}) = r + \lambda\, \sin(\theta_0-\epsilon) > r
\]

\medskip
where the first equality holds in view of (\ref{distp}) and the continuity of $\pi _{\partial T}$, and the second strict inequality holds
recalling that, by the choice of $\epsilon$, the angle $\theta _0 - \epsilon$ belongs to $(0, \pi/2)$. 

We thus have
\begin{equation}\label{contrad2}
d_{\partial S_r}(p_{\lambda})  > r\ . 
\end{equation}

\medskip
Comparing (\ref{contrad1}) and (\ref{contrad2}) we have a contradiction. \qed

\bigskip\bigskip


\bigskip

\section{Proofs of the results in Section \ref{intro}.}\label{secproofs}

For convenience, let us prepone the following remark, which will be useful in the proofs of Theorems \ref{t:c11} and \ref{t:c2}.

\begin{remark}\label{r:par}  Let $S\subset \R ^2$ satisfy (H1), and let $r \in (0, r_S)$. Let $\gamma :[0, L] \to \R^2$ be a local arc-length parametrization of $\partial S_r$ with $\gamma (0) \in C_r(p)$, and denote by $\nu$ the unit normal to $\gamma$ obtained by a counterclockwise rotation of $\pi/2$ of the unit tangent to $\gamma$. By Proposition \ref{l:prox} (ii), the function $\gamma$ is twice differentiable a.e.\ on $[0,L]$; moreover, if we  denote by $ \kappa (s)$ the curvature of $\gamma$ at $\gamma (s)$ (intended as $\langle \gamma'', \nu \rangle$), the function $\kappa$ belongs to $L ^ \infty ([0,L])$. 
If we assume without loss of generality that
\[
\gamma(0) = 0\, , \quad \gamma'(0) = e_1:=(1,0)\, , \quad
p = (0, r) \,,
\]
and we set
\[ \phi(s) := \int_0^s \kappa(t)\, dt\,,
\qquad \forall  s\in [0,L]\, ,
\]
we can write $\gamma$ under the form
\[
\gamma(s) = \left(\int_0^s \cos\phi(t)\, dt\,,\
\int_0^s \sin\phi(t)\, dt\right)\,,
\qquad \forall s\in [0,L]\, .
\]
Indeed, one checks immediately that 
\[
\begin{split}
\gamma'(s) &= (\cos\phi(s),\, \sin\phi(s)),\\
\nu (s) & = (-\sin\phi(s),\, \cos\phi(s)),\\
\gamma''(s) & = \phi'(s) \, (-\sin\phi(s),\, \cos\phi(s)) = \kappa(s) \nu (s).
\end{split}
\]
Accordingly, a local parametrization  of $S$ near $p$ is given by
\[
\eta(s) := \gamma(s) + r \nu (s)
= \left(\int_0^s \cos\phi(t)\, dt-r\, \sin\phi(s)\,,\
\int_0^s \sin\phi(t)\, dt + r\, \cos\phi(s)\right)\,.
\]
In particular, 
one has 
\[
\eta'(s) = (1-r\,\phi'(s)) \, (\cos\phi(s), \sin\phi(s))
= \mu(s) \gamma'(s) \qquad \hbox{ for a.e.}\ s \in[0,L]\, ,
\]
where the function $\mu$ is defined by
\begin{equation}\label{f:mu}
\mu(s) := 1 - r\,\kappa(s) 
\qquad \hbox{ for a.e.}\ s \in[0,L].
\end{equation}
Incidentally, it is worth noticing that the function $\mu$ is nonnegative. Indeed, from \cite[Lemmas~2 and~3]{CFa} we have
$$\kappa (s) \lambda _{S_r} (\gamma (s)) \leq 1 \qquad \hbox{ for a.e.}\ s \in[0,L],$$
which implies $\mu (s) \geq 0$ in view of  (\ref{f:lSr}).
\end{remark}

\bigskip

{\bf Proof of Theorem \ref{t:c11}}. 

Assume that $S$ is not a singleton. Fix $r \in (0, r_S)$, and denote by $S ^*$ the set of points $p \in S$ 
such that
$C_r(p)$ is a semicircumference of radius $r$. 
By Proposition \ref{p:reg},  we know that, for every $p\in S\setminus{S^*}$,
$C_r(p)$ contains exactly two antipodal points.
Moreover, we observe that $S^*$ cannot have accumulation points. Indeed, if $\{ p _n \} \subset S ^*$ is a Cauchy sequence, then, for $n$ and $m$ large enough, $C_r (p_n) \cap B _ r (p _m )\neq \emptyset$, against $d _{\partial S_r} (p _m) = r$. 
We divide the remaining part of the proof in two steps. 

\medskip
\textsl{Step 1:
$S$  is  Lipschitz manifold, with the (possibly empty) set $ S^*$ as boundary.}

\medskip
Let $p \in S$. Since $S$ is not a singleton and it is arc-wise connected, there is an arc of $S$ passing  through $p$. 
Moreover, since ${S^*}$ has no accumulation points, for every $p \in \R ^2$ there exists a ball centered at $p$ which does not intersect $S ^* \setminus \{p\}$, i.e.,  there exists
$\delta > 0$ such that 
$${S^*}\cap B_{\delta}(p) = 
\begin{cases}
\{ p \} & \hbox{ if } p \in S ^ * \\ 
\emptyset &  \hbox{ if } p \in S \setminus S ^* \,.
\end{cases}
$$
Let  
$$\gamma: I \to \R ^ 2\, \qquad \hbox{ with }
I = \begin{cases}
[ 0, \epsilon) & \hbox{ if } p \in S ^ * \\ 
( - \epsilon, \epsilon) &  \hbox{ if } p \in S \setminus S ^* , 
\end{cases}
$$
be a local arc-length parametrization of $\partial S_r$ such that $\pi _ S (\gamma (0)) = p$. 

Choosing $\epsilon$ sufficiently small, and setting 
$$\eta (s) := \gamma (s) + r \nu (s)\, , \qquad \widetilde \gamma (s) := \gamma (s) + 2 r \nu (s) \, , $$
by continuity of the projection map $\pi _S$ and by the choice of $\delta$, 
we may assume that
$$\pi _ S (\gamma (s)) \subseteq B _ \delta (p) \quad \hbox { and }  C_r ( \eta (s) ) =   \big \{ \gamma (s), \widetilde  \gamma (s) \big \} \quad \forall s \in \text{int}\,I \, .$$

In particular, $S \cap B _\delta (p)$ is parametrized by the Lipschitz curve $\eta (s)$, for $s \in I$. 
In order to prove Step 1, we have to show that such a Lipschitz curve is actually the graph of a Lipschitz function. 
To that aim, by possibly decreasing the size of $\epsilon$,
we can further assume that, setting
$R := \min\{r, r_S-r\}$, the curves $\gamma$ and $\widetilde \gamma$ satisfy: 
\begin{equation}
\label{f:rad}
|\gamma(s) - \gamma(t)| < R, \quad
|\widetilde{\gamma}(s) - \widetilde{\gamma}(t)| < R,\quad
|\gamma'(s)-\gamma'(t)| < 1/2,
\qquad
\forall s, t \in I.
\end{equation}
Let us show that, as a consequence of (\ref{f:rad}),   if we choose a system of coordinates 
such that $e_1 = \gamma ' (0)$ and $e _ 2 = \nu (0)$, 
the function
$\eta_1$ is invertible with Lipschitz inverse. 
In fact, let us show that $\eta' _ 1 (s) \geq 1/4$ for a.e. $s \in I$. 
Recall from Remark \ref{r:par} that we have
\[
\eta'(s) = \mu(s)\gamma'(s) \qquad \forall \, s \in I\, ,  
\]
with $\mu$ defined by (\ref{f:mu}). 
By the third condition in (\ref{f:rad}) we readily obtain
\begin{equation}\label{deri1}
\gamma' _ 1 (s) \geq \gamma ' _ 1 (0) - \frac{1}{2} = \frac{1}{2} \qquad \forall s \in I \, .
\end{equation}

\medskip
On the other hand, we claim that
\begin{equation}\label{prodotto}
\pscal{\widetilde{\gamma}'(s)}{\gamma'(s)} \geq 0\, \qquad \forall s \in I \ :\ \nu \hbox{ is differentiable at } s\,.
\end{equation}
Assume by a moment that \eqref{prodotto} holds true. 
Recalling that $\widetilde{\gamma}'(s) = (1 - 2r \kappa(s))\gamma'(s)$,
we obtain the estimate $1-2r\kappa(s) \geq 0$ and hence
\begin{equation}\label{stimamu}
\mu(s)  \geq \frac{1}{2}\qquad
\text{for a.e.}\ s\in I.
\end{equation}

By (\ref{deri1}) and (\ref{stimamu}) we infer that
$$\eta_1'(s)  \geq \frac{1}{4} \qquad \hbox{ for a.e. } s\in I \,. $$

Therefore, the Lipschitz function
$\eta_1$ is invertible with a Lipschitz inverse $\eta_1 ^ {-1}$. 
Then the support of $\eta$ is the graph of the Lipschitz function
$g(x) := \eta_2(\eta_1 ^ {-1}(x))$ (notice that $g$ is defined on a interval of the type $[a, b)$ in case $p \in S ^ *$ and
on an interval of the type $(a, b)$ in case $p \in S \setminus S ^*$). 

We conclude  that $S$ is a $1$-dimensional compact Lipschitz manifold, and that the boundary of such manifold is given precisely by the (possibly empty) set $S ^*$.

Let us go back to the proof of (\ref{prodotto}), which follows by a simple geometrical argument.

\begin{figure}   
\centering   
\def\svgwidth{5cm}   
\begingroup%
  \makeatletter%
  \providecommand\color[2][]{%
    \errmessage{(Inkscape) Color is used for the text in Inkscape, but the package 'color.sty' is not loaded}%
    \renewcommand\color[2][]{}%
  }%
  \providecommand\transparent[1]{%
    \errmessage{(Inkscape) Transparency is used (non-zero) for the text in Inkscape, but the package 'transparent.sty' is not loaded}%
    \renewcommand\transparent[1]{}%
  }%
  \providecommand\rotatebox[2]{#2}%
  \ifx\svgwidth\undefined%
    \setlength{\unitlength}{265.53bp}%
    \ifx\svgscale\undefined%
      \relax%
    \else%
      \setlength{\unitlength}{\unitlength * \real{\svgscale}}%
    \fi%
  \else%
    \setlength{\unitlength}{\svgwidth}%
  \fi%
  \global\let\svgwidth\undefined%
  \global\let\svgscale\undefined%
  \makeatother%
  \begin{picture}(1,1.21093757)%
    \put(0,0){\includegraphics[width=\unitlength]{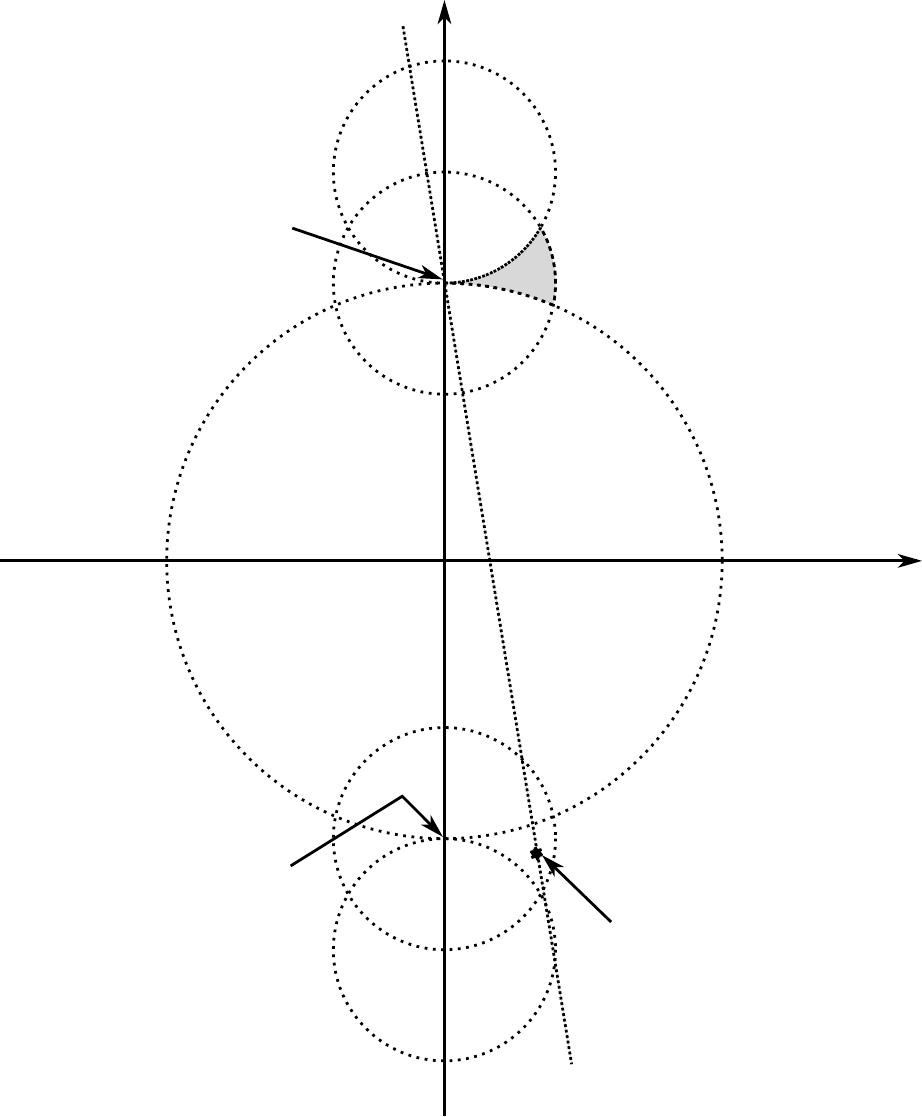}}%
    \put(0.49088393,0.92073045){\color[rgb]{0,0,0}\makebox(0,0)[lb]{\smash{$r$}}}%
    \put(0.49634526,1.03541869){\color[rgb]{0,0,0}\makebox(0,0)[lb]{\smash{$r_S$}}}%
    \put(0.48269193,0.31725179){\color[rgb]{0,0,0}\makebox(0,0)[lb]{\smash{$-r$}}}%
    \put(0.48451235,0.14612965){\color[rgb]{0,0,0}\makebox(0,0)[lb]{\smash{$-r_S$}}}%
    \put(0.67656964,0.19892261){\color[rgb]{0,0,0}\makebox(0,0)[lb]{\smash{$\gamma(t)$}}}%
    \put(0.17321567,0.245344){\color[rgb]{0,0,0}\makebox(0,0)[lb]{\smash{$\gamma(s)$}}}%
    \put(0.16684409,0.98535636){\color[rgb]{0,0,0}\makebox(0,0)[lb]{\smash{$\widetilde{\gamma}(s)$}}}%
  \end{picture}%
\endgroup%
\caption{Proof of \eqref{prodotto}}   
\label{fig-thm11}   
\end{figure}

Namely, let $s\in I$ be fixed so that $\nu$ is differentiable at $s$, and let $t\in I $ denote a generic point, with $t > s$. 
Assume without loss of generality that
$\gamma(s) = (0, -r)$ and $\gamma'(s) = e_1$,
so that $\nu(s) = e_2$, $\eta(s) = 0$,
$\widetilde{\gamma}(s) = (0, r)$
(see Figure~\ref{fig-thm11}, and notice that, to make the remaining of the proof more readable, we are changing system of coordinates with respect to the one chosen above).
Using \eqref{f:rad} and the assumption that $S$ is proximally smooth of radius $r_S$ (and hence $S_r$ is proximally smooth of radius $r_S- r$), 
we get
\begin{equation}
\label{f:region}
\begin{split}
\gamma(t) & \in B_R(0,-r)\setminus \big[B_r(0) \cup B_{r_S - r}(0, -r_S)\big] =:E,\\ \noalign{\smallskip}
\widetilde{\gamma}(t) & \in B_R(0,r)\setminus \big[B_r(0) \cup B_{r_S - r}(0,r_S)\big]= : \widetilde E.
\end{split}
\end{equation}
Indeed, we have $\gamma (t) \in B _R (0, -r)$ by the first condition in \eqref{f:rad}, $\gamma (t) \not \in B _ r (0)$
since $0 \in S$, and finally $\gamma (t) \not \in B_{r_S - r}(0, -r_S)$ by Proposition \ref{l:prox} (iii) combined with the exterior sphere condition recalled in Remark \ref{r:comments} (iv). When $\gamma (t)$ is replaced by  $\widetilde \gamma (t)$, one argues exactly in the same way. 

Notice that, thanks to the inequality $R\leq r$, the regions $E$ and $\widetilde E$ are mutually disjoint. 

\smallskip
We claim that the segments
$[\gamma(s), \widetilde \gamma(s)]$ and
$[{\gamma}(t), \widetilde{\gamma}(t)]$
cannot intersect. 
Namely, assume by contradiction that
\[
[\gamma(s), \widetilde \gamma(s)] \cap
[{\gamma}(t), \widetilde{\gamma}(t)] = \{q\}.
\]
The case $q = p$ is easily
excluded by the fact that
$C_r(\eta(s)) = \{\gamma(s),\, \widetilde{\gamma}(s)\}$.
On the other hand,
if $p\neq q$,
then $d_S(q) = |q-p| \in (0,r)$, so that $q$ must have
a unique projection onto $S$, in contradiction with the
fact that, by construction, both $p$ and $\eta(t)$
are projections of $q$ onto $S$.

Hence, in our coordinate system, the point
$\widetilde{\gamma}(t)$ must lie on the right side of
the line through $\gamma(t)$ and $\widetilde{\gamma}(s)$; hence, in view of (\ref{f:region}), we infer that $\widetilde{\gamma}(t)$ belongs to the set $\widetilde E \cap \{ x_1 > 0 \}$
(corresponding to the shaded region in 
Figure~\ref{fig-thm11}). We conclude that 
\begin{equation}
\label{f:daunlato}
\pscal{\widetilde{\gamma}(t) - \widetilde{\gamma}(s)}{\gamma'(s)}
>  0\,.
\end{equation}
Differentiating (\ref{f:daunlato}) from the right at $t= s$, we obtain (\ref{prodotto}). 

\bigskip
 
\textsl{Step 2: $S$ is of class $C ^ {1, 1}$.} 

\smallskip By Step 1, we know that near each point $p\in S$, $S$ can be parametrized as the graph
of a Lipschitz function $g$.
Since by assumption $S$ is proximally smooth, both 
the epigraph and the hypograph of $g$ are proximally smooth sets. 
Then, by  \cite[Thm.~5.2]{CSW} and \cite[Thm.~6]{Rock82}, 
$g$ is both lower-$C^2$ and  upper-$C^2$, meaning  that $g (s) = \inf _{\tau \in T} G_1(\tau, s)$
and $g (s)  = \sup _{\tau \in T} G_2 (\tau, s)$, where $G_1, G_2$ are continuous in the variable $\tau$  (belonging to some topological space $T$)
and $C^2$  in the variable $s$. It follows that $g$ is locally both semi-concave and semi-convex (see \cite[Prop.~3.4.1]{CaSi}) and, in turn, 
that $g$ is of class $C ^ {1,1}$ (see \cite[Cor.~3.3.8]{CaSi}). 
\qed

\bigskip

{\bf Proof of Theorem \ref{t:c2}}. 

Assume that $S$ is not a singleton.
By Theorem~\ref{t:c11} 
we know that
$S$ is a $1$-dimensional manifold
of class $C^{1,1}$.
We divide the remaining part of the proof in two steps.

\medskip
\textsl{Step 1: $S$ is a manifold without boundary.}

Namely, assume by contradiction that $S$ is a manifold with
boundary.
Let $p$ be a point of this boundary, and let $r \in (0, r_S)$ be fixed. 
Without loss of generality we can assume that 
$p=(0,r)$ and that $C_r(p)$ is the semicircumference
lying in $\{x\leq 0\}$ with endpoints
$a = (0,0)$ and $b = (0, 2r)$.
Let us consider a parametrization $\gamma$ of the
connected component of $\partial S_r$ containing $C_r(p)$
as in Remark \ref{r:par}.
For every $s$
we have that
\[
p(s) := \gamma(s) + r \nu(s) \in S
\]
is equal to $\pi _S(\gamma(s))$.
Moreover, there exists $s_0>0$ such that
\[
C_r(p(s)) = \{\gamma(s), \gamma(s) + 2r \nu(s)\}\,,
\qquad \forall s\in (0, s_0).
\]
We remark that, for $s\in (0, s_0)$, both points in $C_r(p(s))$
must lie in the half-plane $\{x > 0\}$.
In particular one has
\[
\xi(s) := 
\gamma_1(s) + 2r \nu_1(s) =
\int_0^s \cos\phi(t)\, dt - 2r\, \sin\phi(s) > 0
\qquad\forall s\in (0, s_0).
\]
Since $\xi(0) = 0$, this inequality yields
\[
\xi'(0) = 1 - 2r\kappa(0) \geq 0,
\]
that is, $\kappa(0) \leq 1/(2r)$.
On the other hand, since $S_r$ is of class $C^k$ with $k \geq 2$
(see Proposition~\ref{l:prox} (v)), then $\kappa$ is continuous so that
$\kappa(0) = 1/r$, a contradiction.

\medskip
\textsl{Step 2: $S$ is of class $C^2$.}

Let $r\in (0, r_S)$ be fixed. Let $\eta (s)$, for $s \in (s_1, s_2)$ be a local parametrization of $S$. By Step 1 we know that, for every $s \in (s_1, s_2)$ the contact set $C_r (\eta (s))$ consists exactly of two points, say $\gamma (s)$ and $\widetilde \gamma (s)$. We denote by   by  $\Gamma$ and $\widetilde \Gamma$ the support of the two curves $\gamma (s)$ and $\widetilde \gamma (s)$, for $s \in (s_1, s_2)$; moreover, for $i=1,2$, we set $q _i := \gamma (s_i)$, and $\widetilde q _i := \widetilde \gamma (s_i)$. Let $A$ be the open bounded set delimited by 
the two curves $\Gamma$, $\widetilde \Gamma$, 
and the two line segments $[q _1, \widetilde q_1]$, $[q _2, \widetilde q_2]$. 

Since $S$ is proximally $C^2$ of radius $r_S$, by Proposition \ref{l:prox} (v) we have that $\Gamma$ is
of class $C^2$.
Moreover, for any $y = \gamma (s)\in  \Gamma$, consider the line segment
$y + t \nu _ A(y)$, for $t\in [0, 2r]$.
By construction, the mid-point $p:=y+ r \nu _A (y)$ of such segment lies on $S$, while its extremes 
$y$ and $y + 2 r \nu _A (y)$ coincide precisely with the two elements $\gamma (s)$ and $\widetilde \gamma (s)$ of the contact set $C_r (p) = \partial B _ r (p) \cap \partial S _r $.  
We infer that every point in $A$ has a unique projection onto $\Gamma$. 
Then, by using the facts that $\Gamma$ is of class $C ^2$ and that every point in $A$ has a unique projection onto $\Gamma$, we may argue by using the Inverse Function Theorem exactly as done in the proof of 
\cite[Thm.~6.10]{CMf} 
to obtain that $d_{\Gamma}$ is of class $C^2$ on $A$.
Since, by construction, $S\cap A$ agrees with the level set
$\{d_{\Gamma } = r\}\cap A$, by the Implicit Function Theorem
we conclude that $S$ is of class $C^2$.  \qed

\bigskip
\begin{remark}\label{r:inspection}
By inspection of Step 2 in the above proof, one can easily check that the statement of Theorem \ref{t:c2} can be generalized as indicated in Remark \ref{r:comments} (iii). Indeed,  if $S$ satisfies Definition \ref{d:psm} with $C ^2$ replaced by $C ^ {k, \alpha}$, $C ^ \infty$, or $C ^ \omega$, then $\Gamma$ turns out to be of the same class 
$C ^ {k, \alpha}$, $C ^ \infty$, $C ^ \omega$ (cf.\ Proposition \ref{l:prox} (v)). Then, by following the same proof as above (that is, by localizing the 
argument used in \cite[Thm.~6.10]{CMf}) one  concludes that $S$ is of class $C ^ {k, \alpha}$, $C ^ \infty$, $C ^ \omega$, respectively.
\end{remark}

\bigskip
{\bf Proof of Theorem \ref{forBK}}. Clearly, $S$ is a nonempty compact set. Moreover, it is connected ({\it cf.} Proposition 
\ref{top} (iii)), and it has empty interior (otherwise it could not be $S = \high (\Omega)$). We claim that $S$ is proximally $C ^1$. Indeed, by the equality $S = \Sigma (\Omega)$, for every $x \in \Omega \setminus S$ the set $\pi _{\partial \Omega} (x)$ is a singleton, so that $d _{\partial \Omega}$ is differentiable with
$$d' _{\partial \Omega} (x) = \frac{ x - \pi _{\partial \Omega}(x)} {d_{\partial \Omega} (x) } \, \qquad \forall x \in \Omega \setminus S\,.$$
The above equality shows that $d _{\partial \Omega}$ is actually of class $C ^1$ on the set $\Omega \setminus S$, that is, $\partial \Omega$ is proximally $C ^1$ of radius $\rho _\Omega$. By applying (\ref{ugdist}) in Proposition \ref{l:prox} (with $\partial \Omega$ in place of $S$) and letting $r $ tend to $\rho _\Omega$, we obtain 
\begin{equation}\label{f:reldist}
d _ S (x) = \rho _\Omega - d _{\partial \Omega} (x) \qquad \forall x \in \Omega \setminus S\, .
\end{equation}
Hence $S$ is proximally $C^1$, of radius $r _ S \geq \rho _\Omega$.  Then $S$ satisfies (H1) and we can apply Theorem \ref{t:c11} to deduce that $S$ is either a singleton or a $1$-dimensional manifold of class $C ^{1,1}$.  By (\ref{f:ug}), it readily follows that $\Omega = S_{\rho _\Omega}$. In case $\partial \Omega$ is $C ^2$, the function $d _{\partial \Omega}$ is $C ^ 2$ on $\Omega \setminus S$ \cite[Thm.~6.10]{CMf}. Then by (\ref{f:reldist}) $S$ is proximally $C^2$, and the last part of the statement follows from Theorem \ref{t:c2}. \qed

\bigskip
{\bf Proof of Corollary \ref{cor:lambda}}.

Assume by contradiction that $\high (\Omega)\neq \Cut (\Omega)$. Choose two points $x_1$ and $x_2$, with
$ x_1 \in \high (\Omega)$ and $x_2 \in \Cut (\Omega) \setminus \high (\Omega)$, and let
$y _1 \in \pi _{\partial \Omega} (x_1)$, $y _2 \in \pi _{\partial \Omega} (x_2)$. Then
$$\max _{\overline \Omega} d _{\partial \Omega } = \lambda _\Omega (y _1) > \lambda _\Omega (y _2)\ ,$$
against the assumption $\lambda _\Omega$ constant along the boundary.  \qed

\bigskip

{\bf Proof of Theorem \ref{t:nconvex}}.

Since $\Omega$ is a convex set, the distance function
$d_{\partial\Omega}$ is concave in $\overline{\Omega}$,
hence the set $S$ is convex.
Since $S$ does not contain interior points,
the dimension of $S$ (as a convex set) is less than
or equal to $n-1$, i.e., there exists and affine
subspace $V\subset\R^n$ of dimension $\leq n-1$
such that $S\subset V$. 
Let $p,q\in S$ be two points of maximal distance in $S$, i.e.
\[
|p-q| = \text{diam}(S) := \max\{|z-w|;\ w,z\in S\}.
\]
We remark that the hyperplanes through $p$ and $q$ orthogonal to
$p-q$ are support planes to $S$. 


Without loss of generality,  let us assume
that $V = \text{span}\{e_1, \ldots, e_k\}$, $k\leq n-1$, and that
$p = \alpha\, e_1$, $q = -\alpha\, e_1$ for some $\alpha > 0$. 
So we have $\text{diam}(S) = 2\alpha$, and
\begin{equation}\label{f:geoS}
S \subset \big \{x = (x_1, \ldots, x_n):\
|x_1|\leq\alpha,\
x_j = 0\ \forall j=k+1, \ldots, n
\big \}.
\end{equation}

Let us set $W := \text{span}\{e_1, e_n\}$, and let us identify $W$ with $\R ^2$. 
By construction, we have
$$S \cap W = \big \{x= (x_1, x_2):\ x_1\in [-\alpha, \alpha],\ x_2 = 0 \big \}\,.$$
Consider now the convex subset of $\R ^2$ given by
$$A := \Omega\cap W\,.$$
From (\ref{f:geoS}), we infer that the set $A \cap \{ |x_1| \leq \alpha \}$ is given by two line segments 
parallel to $S \cap W$, whereas the set $A \cap \{|x_1| \geq \alpha\}$ is given by two semi-circumferences of radius $\alpha$ centered at $p$ and $q$. Thus $A$ a stadium-like domain, with 
$\Cut(A) = \high(A) = S \cap W\,.$
On the other hand, by the definition of $A$ and the regularity assumption made on $\Omega$, $A$ must have a $C ^2$ boundary.  But the unique stadium-like domain with a $C ^2$ boundary is the disk. So $\alpha = 0$, which means that $S$ has zero diameter, or equivalently is a singleton. 
\qed


\def\cprime{$'$}
\providecommand{\bysame}{\leavevmode\hbox to3em{\hrulefill}\thinspace}
\providecommand{\MR}{\relax\ifhmode\unskip\space\fi MR }
\providecommand{\MRhref}[2]{%
  \href{http://www.ams.org/mathscinet-getitem?mr=#1}{#2}
}
\providecommand{\href}[2]{#2}

\end{document}